\documentclass[a4paper,12pt]{amsart}
\usepackage{amssymb,amscd} 

\title{Links between orthogonal arrays, association schemes and partial geometric designs}

\theoremstyle{plain}

\newtheorem{thm}{Theorem}
\newtheorem{prop}[thm]{Proposition}
\newtheorem{lem}[thm]{Lemma}
\newtheorem{cor}[thm]{Corollary}

\newtheorem{defn}[thm]{Definition}

\newtheorem{rem}[thm]{Remark}
\newtheorem{ex}[thm]{Example}

\newcommand{\mbb}{\mathbb}
\newcommand{\mc}{\mathcal}

\newcommand{\hs}{\hspace{1mm}}

\newcommand{\fl}{\flushleft}
\newcommand{\ti}{\textit}

\newcommand{\flf}{\mathbf}
\newcommand{\F}{\mathbb{F}}

\def \half {\frac12}

\def \tbf{\textbf}

\def \xr {$(X, \{R_i\}_{0 \le i \le d})$}

\def \vhs {\vskip 0.1in}

\def \x {\mathcal{X}}

\def \< {\langle }
\def \> {\rangle }

\usepackage{color}
\definecolor{red}{rgb}{1,0,0}

\begin{document}
\date{\today}
\pagestyle{plain}

\author[K.~Nowak]{Kathleen Nowak}
\address{Department of Mathematics, Iowa State University,
Ames, Iowa, 50011, U. S. A.} \email[K.~Nowak]{knowak@iastate.edu}

\author[O.~Olmez]{Oktay Olmez}
\address{Department of Mathematics,
Ankara University, Tandogan, Ankara, 06100, Turkey}
\email[O. ~Olmez]{oolmez@ankara.edu.tr}

\author[S. Y.~Song]{Sung Y. Song}
\address{Department of Mathematics, Iowa State University, Ames, Iowa, 50011, U. S.
A.} \email[S. Y.~Song]{sysong@iastate.edu}


\maketitle

\begin{abstract}  In this paper, we show how certain three-class association schemes and orthogonal arrays give rise to partial geometric designs. We also investigate the connections between partial geometric designs and certain regular graphs having three or four distinct eigenvalues, three-class association schemes, orthogonal arrays of strength two and particular linear codes. We give various characterizations of these graphs, association schemes and orthogonal arrays in terms of partial geometric designs. We also give a list of infinite families of directed strongly regular graphs arising from the partial geometric designs obtained in this paper. 
\end{abstract}

\section{Introduction}

Partial geometric designs (also known as $1\frac12$-designs) were recently shown to produce directed strongly regular graphs \cite{BOS}. In \cite{Ol} and \cite{NOS} we uncovered which difference sets and difference families produce partial geometric designs.  Here we take the next step and explore the link between these designs and other combinatorial structures. Specifically, we establish connections with strongly regular graphs, certain wreath product association schemes, three-class association schemes, and specific orthogonal arrays of strength two.

It is well-known that many strongly regular graphs give rise to symmetric 2-designs and partial geometric designs (cf. \cite{Br}, \cite{Ne}, \cite{va}). In particular, every complete multipartite regular graph gives rise to a partial geometric design.
Additionally, any strongly regular graph satisfying $\lambda=\mu$ gives rise to a symmetric 2-$(v, k, \lambda)$ design, which is, in turn, a partial geometric design. (See, for example, \cite{Ne, Sp, GS}.)
In fact it is shown that a strongly regular graph with parameters $(v,k,\lambda, \mu)$ gives rise to a partial geometric design if and only if it satisfies either $k=\mu$ or $\lambda=\mu$ (cf. Section \ref{sec-srgs} below).

Every strongly regular graph is realized as a relation graph of some association scheme. In particular, a nontrivial strongly regular graph and its complement are the relation graphs of a two-class association scheme. However, there are graphs that give rise to partial geometric designs but are not realized as relation graphs of association schemes.
Our investigation into finding the source of partial geometric designs begins with studying the characteristics of the graphs that give rise to such designs. We observe that some of these graphs arise as the relation graphs of certain three-class association schemes. This observation leads us to explore the links between partial geometric designs, graphs, and association schemes. Some of these association schemes come from certain orthogonal arrays of strength two and linear codes. As a consequence, we are able to find an infinite family of partial geometric designs and give a list of directed strongly regular graphs arising from these partial geometric designs.

The organization of the paper is as follows. In the following section, we introduce notation that will be used throughout and recall some basic terms from the theory of designs and association schemes.

In Section \ref{sec-srgs}, we characterize the strongly regular graphs that give rise to partial geometric designs.

In Section \ref{sec-wrth}, we recall that the wreath product of an arbitrary association scheme with the trivial association scheme possesses a relation graph isomorphic to a strongly regular graph. Hence such a wreath product association scheme gives rise to a partial geometric design. Conversely, if an imprimitive association scheme of class three or more contains exactly one strongly regular relation graph, then such a scheme must be isomorphic to the wreath product of a scheme with a one-class association scheme.

In Section \ref{sec-assc}, we describe parameter sets of certain three-class association schemes that give rise to partial geometric designs. In particular, we show that if a 3-class symmetric self-dual association scheme of order $3m^2$ satisfies certain parametric conditions, then its adjacency matrices $A_0, A_1, A_2, A_3$ satisfy the following identities for some constants $\alpha_i$ and $\beta_i$:
\begin{eqnarray*}
A_1^3=\beta_1A_1 +\alpha_1(J-A_1),\\
A_2^3=\beta_2A_2+\alpha_2(J-A_2),\\
(A_3+A_0)^3=\beta_3(A_3+A_0)+\alpha_3(J-A_3-A_0).\end{eqnarray*}

In Section \ref{sec-conc}, we then provide concrete examples of such association schemes coming from Hamming codes and certain orthogonal arrays of strength two.

In Section \ref{sec-dsrg}, we provide the parameter sets of directed strongly regular graphs  obtained from the partial geometric designs constructed in this paper by applying the relationship between partial geometric designs and directed strongly regular graphs given by Brouwer-Olmez-Song in \cite{BOS}. Finally, we close with some last remarks on our construction of partial geometric designs.


\section{Preliminaries}
Here we recall some basic facts about block designs and association schemes. We also set the notation that will be used throughout the paper.
\subsection{Designs}

A \textit{block design} is a pair $(P, \mc{B})$ where
 $P$ is a finite set, the elements of which are called \textit{points}, and $\mc{B}$ is a finite collection (possibly multiset) of nonempty subsets of $P$ called \textit{blocks}.

A \ti{tactical configuration}, often also called a \textit{$1$-design}, with parameters $(v,b,k,r)$ is a design $(P,\mc{B})$ with $|P| = v$ and $|\mc{B}| = b$ such that each block consists of $k$ points and each point belongs to $r$ blocks. A 2-$(v,k,\lambda)$ \textit{design} is a 1-design satisfying the added condition that every pair of distinct points is contained in exactly $\lambda$ blocks.


A \textit{partial geometric design} with parameters $(v,b,k,r; \alpha, \beta)$ is a $1$-design $(P, \mc{B})$ with parameters $(v, b, k, r)$ satisfying the `partial geometric' property:
For every point $x \in P$ and every block $B \in \mc{B}$, the number of incident point-block pairs $(y, C)$ such that $y\in B$ and $x\in C$ is $\alpha$ if $x\notin B$ and is $\beta$ if $x\in B$ for some constants $\alpha$ and $\beta$. That is,
\[|\{(y,C) \colon  y\in B\cap C, C\ni x\}|=\left \{ \begin{array}{ll}
\alpha  &\mbox{if } x\notin B,\\
\beta &\mbox{if } x\in B.\\ \end{array}\right . \qquad\]
If $N$ is the point-block incidence matrix of a $(v,b,k, r;\alpha, \beta)$-partial geometric design, then it satisfies
\begin{equation}\label{eq2} JN=kJ, \quad NJ=rJ,\quad  NN^TN=\beta N+\alpha (J-N),\end{equation}
where $N^{T}$ denotes the transpose of $N$, and $J$ is the all-ones matrix.
A 2-$(v,k,\lambda)$ design is partial geometric with $\alpha = k\lambda$ and
$\beta=\lambda\frac{v-1}{k-1}+k\lambda-\lambda$. We will say that a partial geometric design $(P,\mc{B})$ is \ti{symmetric} whenever $v=b$ (and so, $k=r$).
When the design is symmetric, its parameters are simply denoted by $(v,k;\alpha, \beta)$, in short.

In this paper, by the phrase, ``\ti{graph $\Gamma=(V, E)$ gives rise to design $(P, \mc{B})$}," we mean that the adjacency matrix $A$ of $\Gamma$ is \ti{equivalent} to the incidence matrix $N$ of $(P, \mc{B})$.
That is, for each $v\in V$, if we let $N_v=\{x\in V\colon (x, v)\in E\}$ and $\mc{N}=\{N_v\colon v\in V\}$, the pair $(V, \mc{N})$ forms a design that is \ti{isomorphic} to $(P, \mc{B})$.\footnote{There exist bijections $f\colon V\rightarrow P$ and $\phi\colon \mc{N}\rightarrow \mc{B}$ such that $x\in N_v$ if and only if $f(x)\in  \phi(N_v)$.}

\subsection{Association schemes and their Bose-Mesner algebras}

Let $X$ be an $n$-element set, and let
$R_0, R_1, \dots, R_d$ be subsets of $X\times X:=\{(x,y): x,y\in X\}$ with $R_0=\{(x,x): x\in X\}$.
Let $A_i$ be the $n\times n$ \{0, 1\}-matrix representing $R_i$: i.e., \[\left (A_i\right )_{xy}=\left \{ \begin{array}{ll} 1 & \mbox{if } (x,y)\in R_i\\
0 & \mbox{otherwise.}\\ \end{array}\right . \]
The pair $\mc{X}=\left (X, \{R_i\}_{0\le i\le d}\right )$ is called a \ti{$d$-class (symmetric) association scheme} if $A_0, A_1, \dots, A_d$ satisfy the following:
\begin{enumerate}
\item[(1)] $A_0+A_1+\cdots +A_d=J$, where $J$ is the all-ones matrix and $A_0=I$, the identity matrix,
\item[(2)] for each $i\in \{0, 1, \dots, d\}$, $A_i^{T}=A_{i}$,
\item[(3)] for any $h, i, j\in \{0, 1, \dots, d\}$, there exists a constant $p_{ij}^h$ such that
\[A_iA_j=\sum\limits_{h=0}^d p_{ij}^h A_h.\]
\end{enumerate}
The matrices $A_0, A_1, \dots, A_d$ defined above are called the \ti{adjacency matrices} of $\mc{X}$, and the graphs $(X, R_1), (X, R_2), \dots, (X, R_d)$, are called the \ti{relation graphs} of $\mc{X}$.  The constants $p_{ij}^h$ are called the \ti{intersection numbers} of $\mc{X}$, and for any $(x,y)\in R_h$
 \[p_{ij}^h=|\{z\in X: (x,z)\in R_i, \ (z,y)\in R_j\}|.\]
Let $B_i$, $i\in \{0, 1, \dots, d\}$, be the $i$th \ti{intersection matrix} defined by
\[\left (B_i\right )_{jh}=p_{ij}^h.\] Then $B_iB_j=\sum\limits_{h=0}^d p_{ij}^h B_h$.

\vhs

Let $\x =$ \xr\ be an association scheme  with its adjacency matrices $A_0, A_1, \dots, A_d$ and intersection matrices $B_0, B_1, \dots, B_d$. Then the
$\mbb{C}$-space with basis $\{A_0, A_1, \dots, A_d\}$ is an algebra over
the complex numbers, called the {\it Bose-Mesner algebra}  of
$\x$, denoted by $\mc{A}(\mc{X})$ or $\langle A_0, A_1, \dots, A_d\rangle$.
The $\mbb{C}$-algebra generated by $\{B_0, B_1, \dots, B_d\}$ is called the \ti{intersection algebra} of $\mc{X}$.
The Bose-Mesner algebra $\mc{A}(\mc{X})$ and the intersection algebra $\langle B_0, B_1, \dots, B_d\rangle$ are isomorphic $\mbb{C}$-algebras induced by the correspondence $A_i\mapsto B_i$. (For more information, see for example, \cite{BI, BCN}.)


\section{Strongly regular graphs with either $k=\mu$ or $\lambda=\mu$} \label{sec-srgs}
Strongly regular graphs arise from various combinatorial structures, especially in connection with designs and codes. For a complete characterization of partial geometric designs as well as a thorough investigation of their connection to partial geometries and strongly regular graphs, we refer the readers to Bose, Shrikhande and Singh \cite{BSS} and Neumaier \cite{Ne}. In this section, we characterize which strongly regular graphs give rise to symmetric partial geometric designs.

\begin{lem}\label{srg-dsgn}
Let $\Gamma$ be a strongly regular graph with parameters $(v,k,\lambda,\mu)$. Let $A$ be the adjacency matrix of $\Gamma$. Then $A^3=\beta A+\alpha (J-A)$ for some integers $\alpha$ and $\beta$ if and only if either $\lambda=\mu$ or $k=\mu$. (In this case,  $\alpha= (\lambda-\mu)\mu +\mu k$ and $\beta=(\lambda-\mu)^2+k-\mu+(\lambda-\mu)\mu +\mu k$.)
\end{lem}
\begin{proof}
Given a strongly regular graph $\Gamma$ with parameters $(v,k,\lambda,\mu)$, the adjacency matrix $A$ of $\Gamma$ satisfies the identity: \[A^2=kI+\lambda A + \mu (J-I-A).\]
Thus, we have that
\[
A^3=\{(\lambda-\mu)^2+ (\lambda-\mu)\mu +k\mu+k-\mu\}A+(\lambda-\mu)(k-\mu)I+\{(\lambda-\mu)\mu +\mu k\}(J-A).
\]
Therefore, there exist $\alpha$ and $\beta$ such that $A^3=\beta A+\alpha (J-A)$ if and only if
\[(\lambda-\mu)(k-\mu)=0,\ (\lambda-\mu)^2+k-\mu+(\lambda-\mu)\mu +\mu k=\beta \mbox{ and }  (\lambda-\mu)\mu +\mu k=\alpha.\]
Hence the proof follows.
\end{proof}

Every complete multipartite strongly regular graph can be viewed as the complement of $c$-copies of the complete graph $K_n$ on $n$ vertices for some integers $c$ and $n$ (where $c, n\ge 2$).  We denote such a graph by $\overline{cK_n}$.

\begin{cor}
The complete multipartite strongly regular graph $\overline{cK_n}$ gives rise to a symmetric partial geometric design with parameters $(cn, (c-1)n; (c^2-3c+2)n^2, (c^2-3c+3)n^2)$.
\end{cor}
\begin{proof} This strongly regular graph has parameters
$$(v,k,\lambda, \mu)=(cn, (c-1)n, (c-2)n, (c-1)n).$$
The result now follows from Lemma \ref{srg-dsgn}.
\end{proof}
A strongly regular graph with parameters $(v,k,\lambda, \lambda)$ is sometimes called a $(v,k,\lambda)$-graph. The adjacency matrix $A$ of a $(v,k,\lambda)$-graph satisfies identity $A^2=kI+\lambda (J-I)$; therefore, it gives a symmetric 2-$(v,k,\lambda)$-design. Since a symmetric 2-$(v,k,\lambda)$-design is a partial geometric design with parameters $(v,k;k\lambda, k\lambda+k-\lambda)$, so we also have:
\begin{cor}
A $(v,k,\lambda)$-graph gives rise to a partial geometric design with parameters $$(v,k;k\lambda, k\lambda+k-\lambda).$$
\end{cor}
\begin{rem} \label{rem-srg}
\begin{enumerate}
\item[(1)] We note that both the Hamming graph $H(2,4)$ and the Shrikhande graph are $(16, 6, 2)$-graphs. Although these two graphs are non-isomorphic  they give rise to the same 2-$(16, 6, 2)$-design. (cf. \cite[Ch.2 and Ch.4]{CVL}.) Hence, we have a partial geometric design with parameters $(v, k;\alpha, \beta)=(16, 6; 12, 16)$ as all 2-designs are partial geometric.
\item[(2)] We also note that there are many $(v,k,\lambda)$-graphs: Examples of small graphs include $(35, 15, 6)$, $(35, 18, 9)$, $(36, 21, 12)$,  $(45, 12, 3)$, $(63, 32, 16)$ and $(64, 36, 20)$. To see the current list of such strongly regular graphs, visit the homepage of E. Spence \cite{Sp} or A. Brouwer \cite{Br0}.\end{enumerate}\end{rem}

\section{Wreath product of a scheme by a one-class association scheme}\label{sec-wrth}
In this section, we establish a connection between partial geometric designs and wreath products of association schemes. We show that every wreath product association scheme in which one factor is a trivial association scheme gives rise to a partial geometric design. It follows from the fact that such a wreath product association scheme has a relation graph which is strongly regular with $k = \mu$.

Let $\mc{X}=(X, \{R_i\}_{0\le i\le d})$ and $\mc{Y}=(Y,\{S_j\}_{0\le j\le e})$ be association schemes of order $|X|=m$ and $|Y|=n$, respectively. Let $\{A_i\}_{0\le i\le d}$ and $\{C_j\}_{0\le j\le e}$ be the sets of adjacency matrices of $\mc{X}$ and $\mc{Y}$, respectively. Then the adjacency matrices of the wreath product $\mc{X}\wr \mc{Y}$ of $\mc{X}$ and $\mc{Y}$ are \[I_n \otimes A_0, I_n\otimes A_1, \dots, I_n\otimes A_d, C_1\otimes J_m, C_2\otimes J_m, \dots, C_e\otimes J_m,\]
where $A\otimes C=(a_{ij}C)$ denotes the Kronecker product of $A=(a_{ij})$ and $C$.
With this ordering of the adjacency matrices, the relation
matrix of $\mc{X}\wr\mc{Y}$ is given by \[R(\mc{X}\wr\mc{Y})= I_n\otimes R(\mc{X}) + [R(\mc{Y})+d(J_n-I_n)]\otimes J_m,\] where $R(\mc{X})=\sum\limits_{h=0}^d hA_h$ and $R(\mc{Y})=\sum\limits_{h=0}^e hC_h$.

\vhs
Let $\x =$ \xr\ be an association scheme  with its Bose-Mesner algebra $\mc{A}(\mc{X})=\langle A_0, A_1, \dots, A_d\rangle$.
For any relations $R_i$ and $R_j$, define
$$
R_i R_j := \left \{ R_h \colon p_{ij}^h \ne 0 \right \}.
$$
Then, for a nonempty subset $H$ of $\{0, 1, \dots, d\}$, $\{R_h\}_{h \in H}$ is called a {\it closed subset} if
$R_i R_j \subseteq \{R_h\}_{h \in H}$ for any $i, j \in H$.
If $\{R_h\}_{h \in H}$ is a closed subset, then the $\mbb{C}$-space
with basis $\{A_h\}_{h \in H}$ is a subalgebra of $\mc{A}(\x)$, called a {\it Bose-Mesner
subalgebra} of $\x$, denoted by $\mc{A}_H$ or $\mc{A}_H(\mc{X})$.

Let $\x =$ \xr\ be an association scheme with $\mc{A}(\mc{X})=\langle A_0, A_1, \dots, A_d\rangle$, and let $\{R_h\}_{h \in H}$ be a closed subset of $\x$.
Let $\mc{Y} = (Y,\{S_j\}_{0\le j\le e})$ be an association scheme with its Bose-Mesner algebra $\mc{A}(\mc{Y})=\langle C_0, C_1, \dots, C_e\rangle $. Let $\{S_g\}_{g \in G}$ be a closed subset of $\mc{Y}$. We say that the Bose-Mesner subalgebras
$\mc{A}_H(\mc{X})$ and
$\mc{A}_G(\mc{Y})$ are
{\it exactly isomorphic} if there is a bijection $\pi: H \to G$
such that the linear map from
$\mc{A}_H(\mc{X})$ to
$\mc{A}_G(\mc{Y})$  induced by
$A_h \mapsto C_{\pi(h)}$ for $h\in H$ is an algebra isomorphism.

The following properties of the Bose-Mesner algebra of a wreath product
of association schemes (See \cite{Bh, BST}) are useful for our discussion.

\begin{lem} \label{lem-wreath}
\label{lem-wreath}
Let $\x =$ \xr\ be an association scheme. Let $\mc{K}_n$ denote the one-class association scheme whose nontrivial relation graph is the complete graph $K_n$.  If $\x = \mc{Y}\wr \mc{K}_n$ for an association scheme
$\mc{Y} = (Y,\{S_j\}_{0\le j\le e})$ and $\mc{K}_n$, then $e = d-1$ and, by renumbering $R_1, R_2, \dots, R_d$ if necessary, the following
hold.

(i)
$\{R_0, R_1, \dots, R_{d-1} \}$ is a closed subset of $\x$ such that
the Bose-Mesner subalgebra $\mc{A}_{[d-1]}$ with basis $\{A_0, A_1, \dots, A_{d-1}\}$
is exactly isomorphic to the Bose-Mesner algebra of $\mc{Y}$.

(ii) Let $k_0, k_1, \dots, k_d$ denote the valency of $\mc{X}$, and let $m=\sum\limits_{i=0}^{d-1}k_i$. Then $|Y|=m$ and
\begin{equation}
\label{eq-kiad}
A_i A_d = k_i A_d, \ 1 \le i < d;
\end{equation}

\begin{equation}
\label{eq-adsq}
A_d^2 = m(n-1)(J_{mn}-A_d) +m(n-2) A_d
\end{equation}

and
\begin{equation}
\label{eq-adcu}
A_d^3 = m^2 (n^2-3n+2)(J_{mn}-A_d) + m^2(n^2-3n+3)A_d.
\end{equation}
\end{lem}

\begin{proof}[Proof]
Let $C_j$ be the adjacency matrix of $\mc{Y}$ corresponding to $S_j$, $0 \le j \le e$. Then the adjacency matrices $A_i$ of $\mc{Y} \wr \mc{K}_n$ can be expressed as follows:
$$
A_0=I_n \otimes C_0, \ A_1=I_n \otimes C_1,  \dots,
A_{d-1}=I_n \otimes C_e, \ A_d=(J_n-I_n) \otimes J_{m}.
$$
Thus, $e = d-1$ and
\begin{eqnarray*}
A_d^2=((J_n-I_n) \otimes J_{m})^2  & = & (J_n-I_n)^2\otimes J_{m}^2 \\
& = & \left ( (n-2)J_n+I_n\right ) \otimes mJ_{m}\\
& = & m(n-2) (J_n-I_n) \otimes J_{m}  + m (n-1) I_n \otimes J_{m}
     \\
 & = & m(n-2) A_d+m(n-1)(J_{mn}-A_d).
\end{eqnarray*}
This verifies that (\ref{eq-adsq}) holds. We then obtain the identity (\ref{eq-adcu}) by multiplying both sides of (\ref{eq-adsq}) by $A_d$ and using (\ref{eq-adsq}) again to derive the desired form.
It is clear that the valency $k_i$ of $R_i$ is equal to that of $S_i$, for $1 \le i \le d-1$, and (\ref{eq-kiad}) holds.
Obviously, $\{R_1, R_2, \dots, R_{d-1}\}$ is a closed
subset, and the Bose-Mesner subalgebra with basis
$\{A_0, A_1, \dots, A_{d-1}\}$
is exactly isomorphic to the Bose-Mesner algebra of $\mc{Y}$.
\end{proof}

From (\ref{eq-adcu}), we have the following.
\begin{thm}
Let $\mc{X} = \mc{Y} \wr \mc{K}_n$ be the wreath product of association schemes $\mc{Y}$
 and $\mc{K}_n$. Let the adjacency matrices of $\mc{X}$ are ordered such a way that $A_d=(J_n-I_n)\otimes J_m$ where $m=|Y|$. Then the $A_d$ can be viewed as the incidence matrix of a symmetric partial geometric design with parameters
$(mn, m(n-1); m^2(n^2-3n+2), m^2(n^2-3n+3))$.
\end{thm}
\begin{proof}
It immediately follows from the fact that the relation graph $(X, R_d)$ is a multipartite strongly regular graph.
\end{proof}

\vskip 0.5cm

\section{Certain three-class self-dual association schemes of order $3m^2$}\label{sec-assc}
In this section we show that the relation graphs of certain three-class association schemes give rise to partial geometric designs.

In order to represent the parameters of an association scheme in a compact form, we recall the definition of its character table.
Let $\x =$ \xr\ be a symmetric association scheme of order $|X|=n$ with adjacency matrices $A_0, A_1, \dots, A_d$.
Let $E_0=\frac1n J, E_1, \dots, E_d$ denote the primitive idempotents in $\mc{A}(\mc{X})$. Then there are $p_j(i), q_i(j)\in \mbb{C}$ for all  $i, j\in \{0, 1, \dots, d\}$ such that
\[A_j=\sum\limits_{i=0}^d p_j(i)E_i \quad \mbox{and}\quad
E_i=\frac{1}{n} \sum\limits_{j=0}^d q_i(j) A_j.\]
The $(d+1)\times (d+1)$ matrices $P$ and $Q$ whose $(i,j)$-entries are defined by \[P_{ij}=p_j(i)  \ \mbox{ and }\ Q_{ij}=q_j(i)\] are called the \ti{1st eigenmatrix} and \ti{2nd eigenmatrix} of $\mc{X}$, respectively. The first eigenmatrix is often called the \ti{character table} of the association scheme.
We note that $PQ=nI$.
An association scheme is said to be (\ti{formally}) \ti{self-dual} if $P=Q$.
Next, note that if none of the relation graphs of a symmetric three-class association scheme are strongly regular, then every relation graph has four distinct eigenvalues \cite{va}.

\vhs

\begin{lem}\label{lem-z} Let $\mc{Z}$ be a three-class symmetric association scheme of order $3m^2$ for some positive integer $m\equiv 0\pmod{3}$. Then the following two statements are equivalent.
\begin{enumerate}
\item[(1)] The character table $P$ of $\mc{Z}$ is given by
\[P =\left [ \begin{array}{cccc}
1 & m(m-1) & m(m+1)& (m-1)(m+1)\\
1 &  m   &  0&         -m-1  \\
1 & 0   & -m  &  m-1 \\
1 & -m&  m & -1      \\
 \end{array}\right ]. \]
\item[(2)] $\mc{Z}$ is self-dual and its adjacency matrices $A_i$ satisfy the following identities:
\[A_1^3=m^2A_1+\frac13 m^2(m-1)(m-2)J\]
\[A_2^3=m^2A_2+\frac13 m^2(m+1)(m+2)J\]
\[(A_3+I)^3=m^2(A_3+I)+\frac13 m^2(m-1)(m+1)J.\]
\end{enumerate}
\end{lem}

\vhs

\begin{proof} First, it is straightforward to verify that $P^2=3m^2I$, and so $\mc{Z}$ is self-dual.
We can also calculate all the intersection numbers of $\mc{Z}$ directly from the character table by using the basic identity:
\[p^h_{ij}=\frac{1}{n\cdot k_h}\sum\limits_{\nu=0}^3 p_i(\nu)p_j(\nu)p_h(\nu)k_{\nu}\] for $h,i,j\in \{0,1,2,3\}$. Namely, the intersection matrices are given by:
\[B_1 =\left [ \begin{array}{cccc}
0 & 1 & 0& 0\\
m(m-1) & \frac13 m(m-2)& \frac13 m(m-1) & \frac13 m^2-m\\
0 & \frac13 m(m+1) & \frac13 m(m-1)& \frac13 m^2 \\
0 & \frac13 m(m-2) -1 & \frac13 m(m-1)& \frac13 m^2 \\
\end{array}\right ] \]

\[B_2 =\left [ \begin{array}{cccc}
0 &    0 & 1 &        0 \\
0 &    \frac13 m(m+1) & \frac13 m(m-1) & \frac13 m^2 \\
m(m+1)&\frac13 m(m+1) & \frac13 m(m+2)& \frac13 m^2+m \\
0 &   \frac13 m(m+1) & \frac13 m(m+2) -1& \frac13 m^2  \\
\end{array}\right ] \]

\[B_3 =\left [ \begin{array}{cccc}
0 & 0 & 0 & 1\\
0 & \frac13 m(m-2) -1 & \frac13 m(m-1)& \frac13 m^2\\
0 & \frac13 m(m+1)& \frac13 m(m+2) -1& \frac13 m^2 \\
(m-1)(m+1) & \frac13 m(m+1)& \frac13 m(m-1)& \frac13 m^2 -2\\
\end{array}\right ]. \]

Second, by applying the basic identity
\[B_iB_j=\sum\limits_{h=0}^3p^h_{ij}B_h,\] we obtain
\[B_i^3=B_i^2B_i=p^0_{ii}B_0B_i+p^1_{ii}B_1B_i+p^2_{ii}B_2B_i+p^3_{ii}B_3B_i\] for $i=1, 2, 3$, in the intersection algebra of $\mc{Z}$. Note that $B_0B_i=B_i$. By plugging the values of $p^h_{ij}$ in the second identity and using the first identity repeatedly, we obtain the following identities:
\[B_1^3=m^2B_1+\frac13 m^2(m-1)(m-2)(B_0+B_1+B_2+B_3)\]
\[B_2^3=m^2B_2+\frac13 m^2(m+1)(m+2)(B_0+B_1+B_2+B_3)\]
\[(B_3+B_0)^3=m^2(B_3+B_0)+\frac13 m^2(m-1)(m+1)(B_0+B_1+B_2+B_3)\]

Finally, we see that the desired identities are deduced from these identities by the isomorphism between the Bose-Mesner algebra $\langle A_0, A_1, A_2, A_3\rangle$ and intersection algebra $\langle B_0, B_1, B_2, B_3\rangle$. Thus (1) implies (2).

Conversely, if we multiply both sides of each identity in (2) by the all-ones vector $\mathbf{j}$, then for instance, from the first identity, we have
\[A_1^3\mathbf{j}=m^2A_1\mathbf{j}+\frac13 m^2(m-1)(m-2)J\mathbf{j}\]
or equivalently,
\[k_1^3\mathbf{j}=m^2k_1\mathbf{j}+m^4(m-1)(m-2)\mathbf{j}.\]
That is, we have \[k_1^3=m^2k_1+m^4(m-1)(m-2);\] and so, $k_1=m(m-1)$. Similarly, we find $k_2=m(m+1)$ and $k_3+1=m^2$.
Furthermore, since the all-ones matrix $J$ has only one non-zero eigenvalue (which is $3m^2$) and the rest of them are zeros, from the identity $A_1^3-m^2A_1=\frac13 m^2(m-1)(m-2)J$, we see that the remaining possible eigenvalues for $A_1$ are $0, m, -m$. This is also true for the cases of $A_2$ and $A_3+I$. Also we know that by the self-duality of the association scheme, possible multiplicicties for these eigenvalues are $k_1, k_2$, and $k_3$.
Now by using the row- and column-orthogonality of the character table, we can arrange the eigenvalues of $A_1, A_2$ and $A_3$ to obtain $P$. This completes the proof.    \end{proof}

\vhs
As an immediate consequence of this lemma, we have the following.
\begin{thm} \label{thm-z} Let $\mc{Z}$ be a three-class association scheme, and let $A_0, A_1, A_2, A_3$ be its adjacency matrices. Suppose that the character table $P$ of $\mc{Z}$ is given by
\[P =\left [ \begin{array}{cccc}
1 & m(m-1) & m(m+1)& (m-1)(m+1)\\
1 &  m   &  0&         -m-1  \\
1 & 0   & -m  &  m-1 \\
1 & -m&  m & -1      \\
 \end{array}\right ]. \]
Then $\mc{Z}$ gives rise to three symmetric partial geometric designs
coming from the incidence matrices $A_1, A_2 \mbox{ and } A_3+A_0$. In this case,
the parameters $(v, k; \alpha, \beta)$ of corresponding partial geometric designs are given by
\[(3m^2,\ m(m-1);\ \frac13 m^2(m^2-3m+2),\ \frac13 m^2(m^2-3m+5)),\]
\[(3m^2,\ m(m+1);\ \frac13 m^2(m^2+3m+2),\ \frac13 m^2(m^2+3m+5)),\]
\[(3m^2,\ m^2;\ \frac13 m^2(m^2-1),\ \frac13 m^2(m^2+2)).\]
\end{thm}
\vhs
\begin{rem}\label{rem-z} Having a character table of an association scheme is equivalent to having the intersection matrices since the character table essentially generates all intersection numbers and vice versa. However, it is possible for more than one association scheme to have the same character table, and thus, the same intersection numbers. For example, it is well-known that the two strongly regular graphs with the same parameters $(16, 6, 2, 2)$ discussed in Remark \ref{rem-srg}  in Section \ref{sec-srgs}, are relation graphs of two distinct two-class association schemes.
\end{rem}

\section{The association schemes from codes and orthogonal arrays} \label{sec-conc}
In this section, we give concrete examples of the three-class association schemes described in the previous section. Our examples come as fusion schemes of the Hamming association scheme $H(d, 3)$ and are also obtained from a family of linear orthogonal arrays of strength two. As a result, we see an interesting link between the three-class association schemes and linear orthogonal arrays.

\subsection{From Hamming schemes}
Here we show that a three-class fusion scheme of $H(d, 3)$ for each odd $d\ge 3$ gives rise to partial geometric designs.

First we recall the definition of the Hamming scheme $H(d,q)$. Let $S$ be a $q$-element set and let \[V:=S^d=\{(x_1, x_2,\dots, x_d): x_j\in S, \ j=1, 2, \dots, d\}.\]
Define the association relation between any $\flf{x}=(x_1, x_2, \dots, x_d)$ and $\flf{y}=(y_1, y_2,\dots, y_d)\in V$ according to the Hamming distance \[\delta(\flf{x}, \flf{y}):=|\{j\in \{1, 2,\dots, d\}: x_j\neq y_j\}|;\] that is, define \[(\flf{x},\flf{y})\in R_i\ \Leftrightarrow \ \delta(\flf{x}, \flf{y})=i.\]
Then
$(V, \{R_i\}_{0\le i\le d})$ is an association scheme called the $d$-class \ti{Hamming scheme}, over $S$, denoted by $H(d, q)$.

First, we show that among the Hamming schemes of class three, $H(3,3)$ is the only Hamming scheme that gives rise to partial geometric designs.
\begin{prop} \label{ExH(3,3)} The relation graphs of the Hamming scheme $H(3,q)$ give rise to partial geometric designs if and only if $q=3$.
\end{prop}
\begin{proof} The eigenmatrices and the first intersection matrix of $H(3,q)$ are given by
\[P=Q =\left [ \begin{array}{cccc}
1 & 3(q-1) & 3(q-1)^2 & (q-1)^3\\
1 & 2q-3 & (q-1)(q-3) & -(q-1)^2\\
1 & q-3 & -2q+3 & q-1\\
1 & -3 & 3 & -1\\
\end{array}\right ] \]
\[B_1=\left [ \begin{array}{cccc}
0 & 1 & 0 & 0\\
3(q-1) & q-2 & 2 & 0\\
0 & 2(q-1) & 2(q-2) & 3 \\
0 & 0 & q-1 & 3(q-2)\\ \end{array}\right ] \]

By direct calculation, it is shown that the intersection matrices satisfy the identity:
\[B_1^3=3(q^2-3q+2)I + (q^2+3q-3)B_1 +6(q-2)B_2+6B_3.\]
By the algebra isomorphism between the Bose-Mesner algebra and the intersection algebra, it then follows that \[A_1^3=3(q^2-3q+2)I + (q^2+3q-3)A_1 +6(q-2)A_2+6A_3.\]
Therefore, when $q=3$, we have
\[A_1^3=15A_1+6(J-A_1).\]
However, for $H(3, q)$ with $q\neq 3$, there is no way that we can express $A_1^3$ as a linear combination of $A_1$ and $J-A_1$.
Therefore, we see that the first relation graph of $H(3,q)$ gives rise to a partial geometric design if and only if $q=3$. By a similar calculation, we can verify that when $q=3$, \[B_2^3=69B_2+60(B_0+B_1+B_3), \quad (B_3+B_0)^3=33(B_3+B_0)+24(B_1+B_2),\]  or equivalently,
 \[A_2^3=69A_2+60(J-A_2), \quad (A_3+I)^3=33(A_3+I)+24(J-A_3-I).\] This completes the proof.
\end{proof}

We note that $H(3,3)$ has the same parameters as the three-class association scheme described in Lemma \ref{lem-z} with $m=3$.
Although there is no other Hamming scheme whose relation graphs give rise to partial geometric designs, there exists a three-class fusion scheme of $H(d,3)$, for each odd $d\ge 3$, whose relation graphs give rise to partial geometric designs. Kageyama, Saha and Das in \cite[Theorem 2]{KSD} introduced the following three-class fusion scheme $\mc{F}$ of $H(d, 3)$ which will be called the \ti{KSD-scheme} in what follows.


\begin{thm} \label{sch-f}\cite{KSD} Consider Hamming scheme $H(d,3)=(V, \{R_i\}_{0\le i\le d})$ with $d=2l+1$ for  $l\ge 1$, and let $S_0=R_0$ and
\[S_j=\bigcup\limits_{i=0}^{[(d-j)/3]}R_{3i+j}, \mbox{ for } j=1, 2, 3\]
where $[(d-j)/3]$ denotes the greatest integer less than or equal to $(d-j)/3$.
Then $\mc{F}=(V, \{S_0, S_1, S_2, S_3\})$ is a three-class association scheme with the following   intersection matrices:
\small{\[B_1 =\left [ \begin{array}{cccc}
0      &                   1 &                0&               0\\
3^{2l}+(-1)^l3^l & 3^{2l-1}+2(-1)^l 3^{l-1} & 3^{2l-1}+(-1)^l3^{l-1} & 3^{2l-1}+(-1)^l3^l\\
0 &               3^{2l-1}-(-1)^{l}3^{l-1}  & 3^{2l-1}+(-1)^{l}3^{l-1} & 3^{2l-1}\\
0 &              3^{2l-1}+2(-1)^l3^{l-1}-1 & 3^{2l-1}+(-1)^l3^{l-1} & 3^{2l-1}\\
\end{array}\right ] \]

\[B_2 =\left [ \begin{array}{cccc}
0            & 0                     & 1             & 0\\
0 &               3^{2l-1}-(-1)^{l}3^{l-1}  & 3^{2l-1}+(-1)^{l}3^{l-1} & 3^{2l-1}\\
3^{2l}-(-1)^{l}3^l & 3^{2l-1}-(-1)^{l}3^{l-1}& 3^{2l-1}-2(-1)^{l}3^{l-1}& 3^{2l-1}-(-1)^{l}3^l\\
0 & 3^{2l-1}-(-1)^{l}3^{l-1}& 3^{2l-1}-2(-1)^{l}3^{l-1}-1 & 3^{2l-1}\\
\end{array}\right ] \]

\[B_3 =\left [ \begin{array}{cccc}
0            & 0                     & 0            & 1\\
0 &              3^{2l-1}+2(-1)^l3^{l-1}-1 & 3^{2l-1}+(-1)^l3^{l-1} & 3^{2l-1}\\
0 & 3^{2l-1}-(-1)^{l}3^{l-1}& 3^{2l-1}-2(-1)^{l}3^{l-1}-1 & 3^{2l-1}\\
3^{2l}-1 & 3^{2l-1}-(-1)^{l}3^{l-1} & 3^{2l-1}+(-1)^l3^{l-1} & 3^{2l-1}-2\\
\end{array}\right ]. \]} \end{thm}

\begin{rem}\label{rem-f} For every $d$, the association scheme $\mc{F}$ belongs to the family of association schemes $\mc{Z}$ described in  Lemma \ref{lem-z}. In fact, if $l$ is odd, the two schemes $\mc{F}$ and $\mc{Z}$ have the same parameters with $m=3^l$. For each even integer $l$, the parameters of $\mc{F}$ are the same as those of $\mc{Z}$ with $m=3^l$ and the first and second association relations  switched.\end{rem}

\begin{cor}\label{cor-H(d,3)} For each $l\ge 1$, the relation graphs of the association scheme $\mc{F}$ above give rise to three non-isomorphic symmetric partial geometric designs with parameters
\[\big(3^{2l+1},\ 3^{2l}+(-1)^l3^l;\ 3^{4l-1}+(-1)^l3^{3l}+2\cdot 3^{2l-1}, \ 3^{4l-1}+(-1)^l3^{3l}+5\cdot 3^{2l-1}\big),\]
\[\big(3^{2l+1},\  3^{2l}-(-1)^{l}3^l; \ 3^{4l-1}-(-1)^{l}3^{3l}+2\cdot 3^{2l-1},\ 3^{4l-1}-(-1)^{l}3^{3l}+5\cdot 3^{2l-1}\big),\]
\[\big(3^{2l+1}, \ 3^{2l}; \ 3^{4l-1}-3^{2l-1}, \ 3^{4l-1}+2\cdot 3^{2l-1}\big).\]
\end{cor}

\begin{proof} The proof directly follows from the following identities of the intersection matrices:
\[B_1^3=\left ( 3^{4l-1}+(-1)^l3^{3l}+5\cdot 3^{2l-1}\right )B_1+\left ( 3^{4l-1}+(-1)^l3^{3l}+2\cdot 3^{2l-1}\right )(B_0+B_2+B_3),\]
\[B_2^3=\left ( 3^{4l-1}-(-1)^{l}3^{3l}+5\cdot 3^{2l-1}\right )B_2+ \left ( 3^{4l-1}-(-1)^{l}3^{3l}+2\cdot 3^{2l-1}\right )(B_0+B_1+B_3),\]
\[(B_3+B_0)^3=\left(3^{4l-1}+2\cdot3^{2l-1}\right)(B_3+B_0)+\left(3^{4l-1}-3^{2l-1}\right)(B_1+B_2).\] Thus three designs whose incidence matrices are $A_1$, $A_2$ and $A_3+I$ are obtained from this association scheme.
\end{proof}

\subsection{From orthogonal arrays of strength two} Our search for three-class association schemes whose relation graphs give rise to partial geometric designs continues in the context of orthogonal arrays and linear codes. Here we find another way to construct the KSD-schemes using orthogonal arrays of strength 2 coming from suitable linear codes. Taking the runs (codewords) as the elements of the underlying set and defining association relations according to the Hamming distances between the codewords, we obtain three-class association schemes that are isomorphic to KSD-schemes.

Let $S$ be a set of $q$-symbols where $q\ge 2$, and let $I:=\{1, 2, \dots, m\}$. Let $X=S^I$ be the set of all maps from $I$ to $S$. Note that we can view each element $\flf{x}\in X$ as an $m$-tuple $(x_1, x_2, \dots, x_m)$ with symbols $x_i$ in $S$.  A \ti{code} is simply a subset $C$ of $X$. In the case where $S = \mbb{F}_q$ and $C$ forms a vector space over $\mbb{F}_q$, we call $C$ a \ti{linear code}.
Next, an $N$-element subset $Y$ of $X$, viewed as an $N \times m$ array of symbols, is called an \ti{orthogonal array of strength} $t$ and \ti{index} $\lambda$ if every $N \times t$ subarray contains all  possible $q^t$ $t$-tuples exactly $\lambda$ times. Following the notation of \cite{HSS}, we denote an orthogonal array $Y$ with the above parameters by OA$(N, m, q, t)$ where $\lambda=N/q^t$. The rows of an OA$(N, m, q, t)$ are sometimes called the \ti{runs} of the orthogonal array. In what follows, the runs of $Y$ will be denoted $\flf{y}_1, \flf{y}_2, \dots, \flf{y}_N$ with $\flf{y}_i=(y_{i1}, y_{i2}, \dots, y_{im})$. It will be clear whether we consider $Y$ as an $N$-set or as an array from the context.

An orthogonal array is \ti{linear} if it takes a finite field as its symbol set, and its rows form a vector space over the field. We will show that for every positive integer $l$, there exists a linear orthogonal array $OA(3^{2l+1}, 2l+3, 3,2)$ coming from a certain linear $[2l+3, 2l+1]_3$-code which gives an association scheme isomorphic to a KSD-scheme.
First we have the following example.

\begin{ex} \label{OA-(27)}
Two mutually orthogonal Latin cubes of order $3$ give an orthogonal array OA$(27, 5, 3, 2)$.  The transpose of the orthogonal array is expressed as the following $5\times 27$ array $M$.
\[M=\left ( \begin{array}{ccccccccc}
000& 000& 000 & 111& 111& 111& 222& 222& 222\\
000& 111& 222 & 000& 111& 222& 000& 111& 222\\
012& 012& 012 & 012& 012& 012& 012& 012& 012\\
000& 111& 222 & 222& 000& 111& 111& 222& 000\\
012& 120& 201 & 012& 120& 201& 012& 120& 201\\ \end{array}\right )\]

A 3-class association scheme is then obtained as follows:

\fl (i) Let $Y:=\{\flf{y}_i: i=1, 2, \dots, 27\}$ be the set of columns of $M$ (runs of OA$(27, 5, 3, 2)$).

\fl (ii) Let $R_0=\{(\flf{y}_i, \flf{y}_i): \flf{y}_i\in Y\}$, and let
\[\begin{array}{l}
R_1=\{(\flf{y}_i,\flf{y}_j)\in Y\times Y: \delta(\flf{y}_i,\flf{y}_j)=2 \mbox{ or } 5\}\\
R_2=\{(\flf{y}_i,\flf{y}_j)\in Y\times Y: \delta(\flf{y}_i,\flf{y}_j)=3\}\\
R_3=\{(\flf{y}_i,\flf{y}_j)\in Y\times Y: \delta(\flf{y}_i,\flf{y}_j)=4\}\\ \end{array}\]
Then $\mc{Y}=(Y, \{R_i\}_{0\le i\le 3})$ is an association scheme.
Its character table is given by
\[P=Q =\left [ \begin{array}{rrrr}
1 & 6& 8& 12\\
1 & 3 & -4 & 0\\
1 & -3 & -1 & 3\\
1 & 0 & 2 & -3\\
\end{array}\right ], \]
and we have the following identities in the Bose-Mesner algebra of $\mc{Y}$:
$$A_1^3=15A_1 +6(J-A_1),\  (A_2+I)^3=33(A_2+I)+24(J-A_2-I) \mbox{ and }
A_3^3=69A_3+60(J-A_3).$$
\end{ex}

\begin{rem}\label{rem-H(3.3)}
\begin{enumerate}
\item[(1)] The above 3-class association scheme is shown to be isomorphic to $H(3,3)$ discussed in Proposition \ref{ExH(3,3)}. Here the isomorphism is established by the fact that all Hamming schemes except for $H(2,4)$ are uniquely determined by their intersection numbers.
\item[(2)]
This OA$(27, 5, 3, 2)$ can be also obtained as the codewords generated by the three vectors $[1, 0, 0, 1, 1],$ $[0, 1, 0, 0, 1]$ and $[0, 0, 1, 1, 0]$ over $\mbb{F}_3$ in the five dimensional Hamming space $H(5,3)$. 
\end{enumerate}\end{rem}

In the rest of this section,
our alphabet $S$ (the symbol set) will be finite field $\mbb{F}_q$ of order $q$. We will denote the code of \ti{length} $n$ and \ti{size} $N$ over the alphabet $\mbb{F}_q$ by $(n, N)_q$ or by $(n,N,d)_q$ if the \ti{minimum distance} $d$ is known. If the code is an $m$-dimensional subspace of the $n$-dimensional vector space $\mbb{F}_q^n$, we denote it by $[n, m]_q$-code (or $[n,m]$-code if the field is clear from the context). We denote the dual code of a $[n, m]$-code $\mc{C}$, by $\mc{C}^{\perp}$ and its minimum distance by $d^{\perp}$, the dual distance of $\mc{C}$. Two linear codes are \ti{isomorphic} if one can be obtained from the other by permuting the coordinate positions and multiplying each coordinate position by a nonzero element of the field. Two linear orthogonal arrays are considered to be the same if the associated codes are isomorphic as linear codes.

In order to describe how we find linear orthogonal arrays OA$(3^{2l+1}, 2l+3, 3, 2)$ from $[2l+3, 2l+1]_3$-codes, we recall a few useful facts which link orthogonal arrays and linear codes.
R. C. Bose \cite{Bo} explicitly specified how the strength of a linear orthogonal array is determined by the associated code. Ph. Delsarte \cite{De} specified connections between the codes and orthogonal arrays. Our results are based on the following theorem which states a special case of a more profound result due to him. It suffices for our construction of orthogonal arrays as we chiefly concentrate on the linear case. (See \cite[Ch. 4]{HSS} for more information.)

\begin{thm} \cite{De}\label{thm-de} If $\mc{C}$ is a $(n, N, d)_q$ linear code over $\mbb{F}_q$ with dual distance $d^{\perp}$ then the codewords of $\mc{C}$ form the rows of an OA$(N, n, q, d^{\perp}-1)$ with entries from $\mbb{F}_q$. Conversely, the rows of a linear OA$(N, n, q, t)$ over $\mbb{F}_q$ form a $(n,N,d)_q$ linear code over $\mbb{F}_q$ with dual distance $d^{\perp}\ge t+1$. If the orthogonal array has strength $t$ but not $t+1$, $d^{\perp}$ is precisely $t+1$.
\end{thm}

\begin{rem} \label{rem-con} In particular, given $q, m$ and fixed strength $t$, if the $m\times n$ generator matrix $G$ for any $[n,m]_q$-code $\mc{C}$ has the property that every $t$-columns of $G$ are linearly independent vectors in $\mbb{F}_q^m$ over $\mbb{F}_q$, then the $q^m$ codewords of $\mc{C}$ form a linear orthogonal array OA$(q^m, n, q, t)$. For this, we note that $G$ generates $q^m$ codewords all of which become the runs of the $q^m\times n$ orthogonal array, say $M$. The $q^m\times t$ subarray obtained by taking any $t$-columns of $M$ contains linear combinations of the rows of the corresponding $t$-columns of $G$; and it contains each of $q^t$ $t$-tuples of symbols exactly $q^{m-t}$ times.\end{rem}

Thus, in order to obtain orthogonal arrays of strength 2, we simply look at the generator matrices of all $[n,m]_q$-codes over $\mbb{F}_q$ and pick the ones whose dual distance is 3.
By Theorem \ref{thm-de} and Remark \ref{rem-con}, we know that if such a $[2l+3, 2l+1]_3$-code yields an orthogonal array, then it must be a OA$(3^{2l+1}, 2l+3, 3, 2)$. Now we demonstrate an infinite family of orthogonal arrays that give rise to the three-class association schemes we seek.

Consider the $[2l+3,2l+1]_3$-code $\mc{C}$ with generator matrix

$$
G = \left [ \begin{array}{cccccc|cc}
 &  &  &  &  &  & 1 & 1 \\
 &  &  &  & &   & 0 & 1  \\
 &  & & I_{2l+1} &  &  & 1 & 0 \\
 &  &  &  &  &  & 0 & 0 \\
 & &  & &  &  &\vdots&\vdots \\
 &  &  &  &  &  & 0 & 0 \\
 \end{array}
\right]
$$
where $I_{2l+1}$ denotes the $(2l+1)\times (2l+1)$ identity matrix. Then the dual code $\mc{C}^{\perp}$ is generated by
$$
G^{\perp} = \left[ \begin{array}{cccccccc}
1 & 0 & 1 & 0 & \cdots & 0 & 2 & 0 \\
0 & 1 & 2 & 0 & \cdots & 0 & 1 & 2
\end{array}
\right]
$$
with its weight distribution $(1, 0, 0, 4, 2, 2, 0, \cdots, 0)$.
Thus, by Theorem 16, the codewords of the $(2l+3, 3^{2l+1})_3$-code $\mc{C}$ with dual distance $3$, form the rows of an $OA(3^{2l+1},2l+3,3,2)$.

Next, consider the $[2l+3,2l+1]_3$-codes $\mc{C}'$ with generator matrix
$$
G' = \left [ \begin{array}{cccccc|cc}
 &  &  &  &  &  & 0& 0 \\
 &  &  &  & &   & 0 & 0  \\
 &  & & I_{2l+1} &  &  & \vdots & \vdots \\
 &  &  &  &  &  &  & \\
 &  &  &  &  &  & 0 & 0 \\
 \end{array}
\right].
$$
We can identify $\mc{C}'$ as the Hamming space $H(2l+1,3)$ by viewing $\mc{C}'$ as a natural embedding of $\F_3^{2l+1}$ in $\F_3^{2l+3}$. Moreover, suppose we define
$$M_0 = \{\textbf{0}\},\quad M_i = \{\textbf{x} \in \mc{C}\setminus \{\textbf{0} \} \hs \colon \hs \delta_H (\textbf{0}, \textbf{x}) \equiv i (\text{mod } 3) \}\ \mbox{ for } i = 1, 2, 3,$$ and the sets $M'_i$ for $\mc{C}'$ in the same manner. Then $\{M_i \hs \colon \hs i = 0, 1, 2, 3\}$ forms a partition of $\mc{C}$. In the same spirit, $\{M'_i \hs \colon \hs i = 0, 1, 2, 3\}$ forms a partition of $\mc{C'}$ according to their Hamming weights. By Theorem \ref{sch-f}, we know that $$\{|M'_i| \hs \colon \hs i = 0, 1, 2, 3\} = \{1, 3^l(3^l+1), 3^l(3^l-1),3^{2l}-1\}.$$
By establishing a vector space isomorphism between $\mc{C}'$ and $\mc{C}$ as below, we can also see that $$\{|M_i| \hs \colon \hs i = 0, 1, 2, 3\} = \{1, 3^l(3^l+1), 3^l(3^l-1),3^{2l}-1\}.$$

For this, let $\flf{e}_1, \flf{e}_2, \dots, \flf{e}_{2l+1}$ denote the rows of $G'$, as the basis vectors for $\mc{C}'$, and $\tbf{r}_1, \tbf{r}_2, \dots, \tbf{r}_{2l+1}$ denote the rows of $G$, which form a basis for $\mc{C}$. Define a map $\phi: \mc{C'} \rightarrow \mc{C}$ by
$$
\phi(\flf{e}_1) = \tbf{r}_1+\tbf{r}_2+\tbf{r}_3, \hs \phi(e_2) = \tbf{r}_2, \hs \phi(e_3) = \tbf{r}_3,
$$
and for $i >3$,
$$
\phi(\flf{e}_i) =
\begin{cases}
\tbf{r}_i + \tbf{r}_{i+1}, & \text{if } i \equiv 0(\text{mod }2) \\
\tbf{r}_{i-1}+2\tbf{r}_i, & \text{if } i \equiv 1(\text{mod }2)
\end{cases}.
$$
It is clear that $\phi$ is a vector space isomorphism.  Furthermore, this map $\phi$ maps $M'_i$ to $M_i$ setwise; namely, $\phi(M'_1) = M_2$, $\phi(M'_2) = M_1$ and $\phi(M'_3) = M_3$. For this,
we recall that the sets $M'_i$ and $M_i$ were defined according to the Hamming weight of the codewords. Notice that a weight-$s$ codeword $\tbf{x}$ can be expressed as $\sum\limits_{h=1}^s \alpha_h\flf{e}_{j_h}$ for some $\alpha_h\in \mbb{F}_3^{*}$ and some $s$-set $\{j_1, j_2, \dots, j_s\}$ with $1\le j_1<j_2<\cdots <j_s\le 2l+1$ (where $\mbb{F}_3^{*}=\mbb{F}_3-\{0\}$). So, we can express
$$
M'_i = \bigg\{\sum_{h = 1}^{s} \alpha_h \flf{e}_{j_h} \hs \colon \hs {1\le j_1<j_2<\cdots <j_s\le 2l+1},\ \alpha_h \in \F_3^*,\  {s\equiv i(\mbox{mod} {3})},\ 1\le s\le 2l+1\bigg\}.
$$
Then, by direct computation, it can be verified that
$$
\phi(M'_i) = \bigg\{\sum_{h = 1}^{s} \alpha_h \phi(\flf{e}_{j_h})\hs\colon\hs  \sum_{h = 1}^{s} \alpha_h \flf{e}_{j_h}\in M'_i\bigg\}\subseteq
\begin{cases}
M_2 & \text{if } i = 1 \\
M_1 & \text{if } i = 2 \\
M_3 & \text{if } i = 3 \\
\end{cases}.
$$
So, it follows that
\[\phi(M'_1) = M_2,\quad \phi(M'_2) = M_1,\quad \phi(M'_3) = M_3\] as $|M'_1| +|M'_2| +|M'_3| = |M_1|+|M_2|+|M_3| = 3^{2l+1}-1$.

As a consequence, we have the following.
\begin{thm} \label{thm-OA}
\begin{enumerate}
\item For each $l\in \mbb{N}$, there exists a linear $[2l+3, 2l+1]_3$-code $\mc{C}$ with dual distance 3, such that $\mc{C}$ is partitioned into $M_0, M_1, M_2, M_3$ where $M_0=\{\textbf{0}\}$ and for $i=1, 2, 3$, $M_i=\{\tbf{x} \in \mc{C}\setminus\{\tbf{0}\} \colon \delta_H(\textbf{0}, \textbf{x})\equiv i \pmod{3}\}$
with cardinalities $|M_i|\in \{3^{2l}-3^l, 3^{2l}+3^l, 3^{2l}-1\}$. In this case, any linear $[2l+3, 2l+1]_3$-code equivalent to $\mc{C}$ is an orthogonal array OA$(3^{2l+1}, 2l+3, 3, 2)$.

\item Defining relations on $\mc{C}$ by
\begin{align*}
R_0 &=\{(\tbf{x}, \tbf{x}) \hs \colon \hs \tbf{x} \in \mc{C} \} \\
R_1&=\{(\tbf{x},\tbf{y})\in \mc{C}\times \mc{C} \hs : \hs   \delta_H(\tbf{x}, \tbf{y}) \equiv 1 (\text{mod }3)\}\\
R_2 &=\{(\tbf{x},\tbf{y})\in \mc{C}\times \mc{C} \hs : \hs  \delta_H(\tbf{x}, \tbf{y}) \equiv 2 (\text{mod }3)\}\\
R_3 &=\{(\tbf{x},\tbf{y})\in \mc{C}\times \mc{C} \hs : \hs \tbf{x} \neq \tbf{y},  \delta_H(\tbf{x}, \tbf{y}) \equiv 0 (\text{mod }3)\}
\end{align*}
we obtain a three-class association scheme $\mc{W}=(\mc{C}, \{R_i\}_{0\le i\le 3})$
which has the same parameters as those for the KSD-scheme $\mc{F}$ for $d=2l+1$ defined in Theorem \ref{sch-f}.
\end{enumerate}
\end{thm}
\begin{proof} The statement (1) is summary of what we have discussed earlier. For statement (2),
we recall that, under the map $\phi$, all codewords of weight $i$ in $\mc{C}'$ are mapped to the codewords of weight $2i \pmod{3}$ in $\mc{C}$; and thus, for any $\tbf{x}, \tbf{y}\in \mc{C}'$, $\delta_H(\tbf{x}, \tbf{y})\equiv i \pmod{3}$ if and only if $\delta_H(\phi(\tbf{x}), \phi(\tbf{y}))\equiv 2i \pmod{3}$.
Therefore, the three-class association scheme $\mc{W}$ defined on $\mc{C}$ and the KSD-scheme $\mc{F}$ defined on $H(2l+1, 3)$ in Theorem \ref{sch-f} share the same parameter sets. That is, if the parameters of KSD scheme are $p_{ij}^h$, then those for $\mc{W}$ are $p_{\sigma(i) \sigma(j)}^{\sigma(h)}$ where $\sigma=(12)$ is the transposition in $S_3$.
\end{proof}

We now give an example to illustrate what we have discussed in this subsection.

\begin{ex} \label{ex-[7,5]} The linear $[7,5]_3$-code $\mc{C}$ generated by
\[ \left(\begin{array}{rrrrrrr}
1 & 0 & 0 & 0 & 0 & 1 & 1 \\
0 & 1 & 0 & 0 & 0 & 0 & 1 \\
0 & 0 & 1 & 0 & 0 & 1 & 0 \\
0 & 0 & 0 & 1 & 0 & 0 & 0 \\
0 & 0 & 0 & 0 & 1 & 0 & 0 \\
\end{array}\right)\]
has weight distribution $(1,4,8,24,60,82,56,8)$ and dual weight distribution $(1,0,0,4,2,2,0,0)$.
This code $\mc{C}$ (and any code equivalent to $\mc{C}$) gives OA$(243, 7, 3, 2)$. By defining association relations by
\begin{eqnarray*}
R_0 &=& \{(x,x) \hs | \hs x \in X\}\\
R_1 &=& \{(x,y) \hs | \hs \delta_{H}(x,y) \in \{1,4,7\}\} \\
R_2 &=& \{(x,y) \hs | \hs \delta_{H}(x,y)  \in \{2,5\} \} \\
R_3 &=& \{(x,y) \hs | \hs \delta_{H}(x,y) \in \{3,6\}\},
\end{eqnarray*}
we obtain a three-class association scheme, described as $\mc{W}$.
The intersection matrices of $\mc{W}$ are given by
\[B_1 =\left [ \begin{array}{cccc}
0 & 1 & 0& 0\\
72 & 21 & 24 & 18\\
0& 30& 24 & 27\\
0 & 20 & 24 &27 \\
\end{array}\right ] \quad
B_2 =\left [ \begin{array}{cccc}
0 & 0 & 1& 0\\
0 & 30 & 24& 27\\
90& 30& 33 & 36\\
0 & 30 & 32 & 27\\
\end{array}\right ] \quad
B_3 =\left [ \begin{array}{cccc}
0 & 0& 0& 1\\
0 & 20 & 24 & 27\\
0& 30& 32 & 27\\
80& 30 & 24 & 25\\
\end{array}\right ]. \]
and for this association scheme, we have the following identities in the Bose-Mesner algebra of $\mc{W}$:
\begin{eqnarray*}
A_1^3 &=& 1593A_1 +1512(J-A_1),\\
A_2^3 &=& 3051A_2+2970(J-A_2),\\
(A_3+I)^3 &=& 2241(A_3+I)+2160(J-A_3-I).\\
\end{eqnarray*}
Therefore, we obtain three partial geometric designs with parameters
\[(v, k; \alpha, \beta)=(243, 72; 1512, 1593), (243, 90; 2160, 2241), (243, 81; 2970, 3051).\]
\end{ex}

\section{Directed strongly regular graphs}\label{sec-dsrg}

Given a partial geometric design, Theorems 2.1 and 2.2 in \cite{BOS} tell us how to construct two directed strongly regular graphs. One is defined on the flags of the design, while the other is defined on the antiflags. Here we list the parameters of the directed strongly regular graphs that we can obtain from the partial geometric designs constructed above.

\begin{defn}
A \textit{directed strongly regular graph} (DSRG) with parameters $({\tt v}, {\tt k}, {\tt t}, \lambda', \mu')$ is a directed graph on ${\tt v}$ vertices without loops such that
\begin{itemize}
\item[(1)] Every vertex has in-degree and out-degree ${\tt k}$,
\item[(2)] Every vertex has ${\tt t}$ out-neighbors which are also in-neighbors, and
\item[(3)] For any two distinct vertices $x$ and $y$, the number of directed paths from $x$ to $y$ of length $2$ is $\lambda'$ if $ x \rightarrow y$ and is $\mu'$ otherwise.
\end{itemize}\end{defn}

We observe that the adjacency matrix of a DSRG with parameters $({\tt v, k, t}, \lambda', \mu')$ has the property that  \[AJ=JA={\tt k}J \ \mbox{ and }\ A^2={\tt t}I+\lambda' A+\mu'(J-I-A).\]

\begin{thm}\cite{BOS}
 Let $(P,\mc{B})$ be a $1$-design. The following three statements are equivalent.
\begin{enumerate}
\item $(P, \mc{B})$ is a partial geometric design.
\item The  directed graph $\Gamma$ defined by  $V(\Gamma) = \{(p,B) \in P \times \mc{B} \colon p \in B\}$ with the adjacency
 $$
 (p,B) \rightarrow (q,C) \text{ if and only if } (p,B) \neq (q,C) \text{ and }p \in C,
 $$
  is a DSRG.

\item The directed graph $\Gamma'$ defined by  $V(\Gamma') = \{(p,B) \in P \times \mc{B} \colon p \notin B\}$ with adjacency
 $$
 (p,B) \rightarrow (q,C) \text{ if and only if } p \in C,
 $$
 is a DSRG.
\end{enumerate}
\end{thm}
\vhs

\begin{rem}\label{rem-dsrg} According to this theorem, any symmetric partial geometric design with parameters $(v, k; \alpha, \beta)$ gives rise to two DSRGs whose parameters $({\tt v, k, t}, \lambda', \mu')$ are given by
\[(v(v-k),\ k(v-k),\ k^2-\alpha,\ k^2-\beta,\ k^2-\alpha) \mbox{ and } (vk,\ k^2-1,\ \beta-1,\ \beta-2,\ \alpha).\]
Therefore, from the partial geometric designs we have obtained in this paper, (in Theorem \ref{thm-z}, Corollary \ref{cor-H(d,3)} and Theorem \ref{thm-OA}) we obtain the DRSGs with the following parameters.
\begin{table}
\begin{center}
\caption{Parameters of DSRGs obtained from partial geometric designs}
(Here $m=3^l$ for every positive integer $l$.)
\label{Table 1}
{\small{
\[\begin{array}{|c|c|c|c|c|}\hline
&&&&\\
{\tt v} & {\tt k} & {\tt t} & \lambda' & \mu'\\
\hline
&&&&\\
3m^3(2m+1)& m^2(m-1)(2m+1)& \frac13 m^2(2m^2-3m+1)& \frac13 m^2(2m^2-3m-2)& \frac13 m^2(2m^2-3m+1)\\
&&&&\\
3m^3(m-1)& m^2(m-1)^2-1& \frac13 m^2(m^2-3m+5)-1& \frac13 m^2(m^2-3m+5)-2 & \frac13 m^2(m^2-3m+2)\\
&&&&\\ \hline
&&&&\\
3m^3(2m-1)& m^2(m+1)(2m-1)& \frac13 m^2(2m^2+3m+1)& \frac13 m^2(2m^2+3m-2)& \frac13 m^2(2m^2+3m+1)\\
&&&&\\
3m^3(m+1)& m^2(m+1)^2-1& \frac13 m^2(m^2+3m+5)-1& \frac13 m^2(m^2+3m+5)-2& \frac13 m^2(m^2+3m+2)\\
&&&&\\ \hline
&&&&\\
3m^2(2m^2+1)& (m^2-1)(2m^2+1)&  \frac13 (2m^2-1)(m^2-1)& \frac13 (2m^4-8m^2+3) & \frac13 (2m^2-1)(m^2-1)\\
&&&&\\
3m^2(m^2-1) & (m^2-1)^2-1& \frac13 m^2(m^2+2)-1& \frac13 m^2(m^2+2)-2& \frac13 m^2(m^2-1)\\
&&&&\\ \hline
\end{array}\]}}
\end{center}
\end{table}
\end{rem}


\section{Concluding remarks}
We examined graphs, association schemes, and orthogonal arrays as possible sources of partial geometric designs. At the end of our search we focused on the orthogonal arrays associated with a particular family of three-class association schemes. However, as we mentioned in the Introduction, there are graphs that give rise to partial geometric designs but are not realized as relation graphs of association schemes. In fact, there are many such graphs obtained from orthogonal arrays. On the set of runs of an orthogonal array, by defining the adjacency of any two codewords according to their Hamming distance, we obtain many graphs that give rise to partial geometric designs. So as not to extend this paper much further, we shall close with just one example. As such examples are abundant, further research is required on the topic.
\begin{ex} Consider the $[7,5]$-code $\mc{C}$ over $\mathbb{F}_3$ with generator matrix:
\begin{center}		
$ \left(\begin{array}{rrrrrrr}
1 & 0 & 0 & 0 & 0 & 2 & 1 \\
0 & 1 & 0 & 0 & 0 & 1 & 1 \\
0 & 0 & 1 & 0 & 0 & 1 & 1 \\
0 & 0 & 0 & 1 & 0 & 1 & 0 \\
0 & 0 & 0 & 0 & 1 & 1 & 0 \\
\end{array}\right)$
\end{center}
This code $\mc{C}$ has weight distribution $(1, 0, 12, 34, 42, 96, 46, 12)$ and dual weight distribution $(1, 0, 0, 2, 0, 0, 6, 0)$. Hence it gives an OA$(3^5, 7, 3, 2)$ with $\lambda=3^3$.
As before, the code $\mc{C}$ (or any code that are equivalent to $\mc{C}$) define $X = \mc{C}$ and
\begin{eqnarray*}
R_0 &=& \{(x,x) \hs | \hs x \in X\}\\
R_1 &=& \{(x,y) \hs | \hs \delta_{H}(x,y) \in \{1,4,7\}\} \\
R_2 &=& \{(x,y) \hs | \hs \delta_{H}(x,y)  \in \{2,5\} \} \\
R_3 &=& \{(x,y) \hs | \hs \delta_{H}(x,y) \in \{3,6\}\}
\end{eqnarray*}
Then the partition $\{R_i\}_{0 \leq i \leq 3}$ of $X\times X$ does not form an association scheme. However, the adjacency matrices $A_i$ of the graphs $(X, R_i)$ for $i=1, 2, 3$ satisfy the following identities:
\begin{eqnarray*}
A_1^3 &=&1215A_1 +486(J-A_1),\\
A_2^3&=& 5589A_2+4860(J-A_2),\\
(A_3+I)^3&=& 2673(A_3+I)+1944(J-A_3-I).
\end{eqnarray*}
Thus, taking each of $A_1, A_2$ and $A_3+I$ as the incidence matrix of a symmetric design, we get three partial geometric designs with parameters
\[(v, k; \alpha, \beta)=(243, 54; 486, 1215), (243, 108; 4860, 5589), (243, 81; 1944, 2673).\]
\end{ex}


\end{document}